\title{
Sharp Entropy Bounds for Plane Curves and Dynamics of the Curve Shortening Flow
}
\author{
Julius Baldauf and Ao Sun
}

\date{
\small \today
}
\documentclass[10pt, letterpaper, leqno]{article}

\usepackage{amsmath}
\usepackage{amsthm,xcolor}
\usepackage[colorlinks, linkcolor=red]{hyperref}
\usepackage[noabbrev,nameinlink]{cleveref}

\theoremstyle{plain}

\newtheoremstyle{named}%
    {}{}{\itshape}{}{\bfseries}{.}{.5em}{\thmnote{#3}}
\theoremstyle{named}
\newtheorem*{namedtheorem}{Theorem}

\newcommand\bref[3][blue]{%
    \begingroup%
    \hypersetup{linkcolor=#1}%
    \hyperlink{#2}{#3}%
    \endgroup}

\usepackage[margin=1.5in]{geometry}
\usepackage[english]{babel}
\usepackage{amssymb}
\usepackage{mathrsfs}
\usepackage{mathtools}
\usepackage{bbm}
\usepackage{amsthm}
\usepackage{fancyhdr}
\usepackage{tikz-cd}
\usepackage{comment}
\usepackage{enumitem}
\usepackage{soul} 
\hypersetup{
    colorlinks=true,
    linkcolor=blue,
    filecolor=magenta,      
    urlcolor=cyan,
	citecolor=blue
}

\numberwithin{equation}{section}

\theoremstyle{plain} 
\newtheorem{lemma}[equation]{Lemma}

\newtheorem{prop}[equation]{Proposition}
\newtheorem{thm}[equation]{Theorem}
\newtheorem*{thmA*}{Theorem A}
\newtheorem*{thmB*}{Theorem B}
\newtheorem{cor}[equation]{Corollary}

\theoremstyle{definition} 
\newtheorem{remark}[equation]{Remark}

\newtheorem{defn}[equation]{Definition}

\makeatletter
\let\c@defn\c@thm
\makeatother

\makeatletter
\let\c@remark\c@thm
\makeatother

\makeatletter
\let\c@prop\c@thm
\makeatother

\makeatletter
\let\c@lemma\c@thm
\makeatother

\makeatletter
\let\c@cor\c@thm
\makeatother

\makeatletter
\let\c@example\c@thm
\makeatother

\newcommand{\R}{\mathbb{R}}
\newcommand{\Z}{\mathbb{Z}}
\newcommand{\N}{\mathbb{N}}

\newcommand{\inv}[1]{#1^{-1}}

\renewcommand{\tilde}{\widetilde}

\newcommand{\be}{\begin{equation}}
\newcommand{\ee}{\end{equation}}

\newcommand{\ind}{\mathrm{ind}}

\begin{document}
\maketitle
\begin{abstract}
We prove that a closed immersed plane curve with total curvature $2\pi m$ has entropy at least $m$ times the entropy of the embedded circle, as long as it generates a type I singularity under the curve shortening flow (CSF). We construct closed immersed plane curves of total curvature $2\pi m$ whose entropy is less than $m$ times the entropy of the embedded circle. As an application, we extend Colding-Minicozzi's notion of a generic mean curvature flow to closed immersed plane curves by constructing a piecewise CSF whose only singularities are embedded circles and type II singularities. 
\end{abstract}


\section{Introduction}

\paragraph{}

Huisken \cite{Il} conjectured that for mean curvature flow (MCF) from generic initial embedded hypersurfaces, all singularities are spheres or cylinders. In a recent fundamental paper \cite{CoMi}, Colding and Minicozzi made an important step towards establishing Huisken's genericity conjecture. In that paper, they define the {\it entropy} of an immersed hypersurface $\Sigma\subset \R^{n+1}$ to be 
\be \label{eq:entropy definition}
\lambda(\Sigma)=\sup_{x_0,t_0}\;\;\;(4\pi t_0)^{-n/2}\int_{\Sigma}e^{-\frac{|x-x_0|^2}{4t_0}}d\mu, 
\ee where the supremum is taken over all translations $x_0\in \R^{n+1}$ and rescalings $t_0>0$ of $\Sigma$. Entropy is nonincreasing along MCF and is constant along the self-shrinking flows, which are known as {\it shrinkers}. Entropy is used in the following important argument: if $\Sigma$ is a shrinker arising as the limit of a blowup sequence of a singular point of the MCF starting at the initial hypersurface $\Sigma_0$ then $\lambda(\Sigma)\leq \lambda(\Sigma_0)$. Entropy thus provides a vital tool for ruling out certain singularities for a MCF.

The first main theorem of this paper gives entropy lower bounds for plane curves in terms of a topological quantity, the {\it turning number}, which is the total curvature divided by $2\pi$. We denote by $\Gamma_m\subset \R^2$ the {\it $m$-covered circle} with radius $\sqrt{2}$, i.e.\ $\Gamma_m$ is an immersion $S^1\to \R^2$ whose turning number is $m$ and whose image is the circle of radius $\sqrt{2}$. It is easy to see that $\Gamma_m$ is a shrinker for all $m\geq 1$.

\begin{thmA*}
Let $\Gamma\subset \R^2$ be a closed immersed curve with turning number $m$.
\begin{enumerate}[label=\textbf{A.\arabic*},ref=A.\arabic*]
\item\label{thm:main theorem type I singularity then entropy has lower bound}
If the CSF starting at $\Gamma$ generates a type I singularity, then 
$\lambda(\Gamma)\geq\lambda(\Gamma_m)$.
\item\label{thm:main theorem existence of perturbation reduce the entropy}
There exist closed immersed curves $\Gamma'\subset \R^2$ with turning number $m$, but
$\lambda(\Gamma')<\lambda(\Gamma_m).$ 
Any such curve generates only type II singularities under the CSF.
\end{enumerate}
\end{thmA*}

Straightforward calculations show that
\be\label{eq:entropy of m-covered circle}
\lambda(\Gamma_m)=m\cdot\lambda(\Gamma_1)=m\sqrt{2\pi/e}.
\ee
In other words, the entropy of the $m$-covered circle is $m$ times the entropy of the embedded circle. 

An important special case of Theorem \ref{thm:main theorem type I singularity then entropy has lower bound} is that $\Gamma\subset \R^2$ is a closed shrinker of turning number $m$. Closed shrinkers for the CSF are called {\it Abresch-Langer curves} \cite{AbLa}. For any pair of relatively prime integers $m,n\in \N$ satisfying \be \frac{1}{2}<\frac{m}{n}<\frac{\sqrt{2}}{2},\ee there exists a closed convex shrinker $\Gamma_{m,n}$ with turning number $m$ and whose curvature has $2n$ critical points. 
Moreover, every closed shrinker is of this form (other than the round circles). 
Theorem \ref{thm:main theorem type I singularity then entropy has lower bound} implies that 
\be
\lambda(\Gamma_{m,n})\geq\lambda(\Gamma_p)
\ee
for every Abresch-Langer curve $\Gamma_{m,n}$ and every $p$-covered circle $\Gamma_p$ with $m\geq p$. The normalization constant $(4\pi)^{-n/2}$ in the definition of entropy (\ref{eq:entropy definition}) is chosen so that hyperplanes have entropy $1$. Moreover,
\be\label{eq:product property of entropy}
\lambda(\Sigma)=\lambda(\Sigma\times \R)
\ee 
for all hypersurfaces $\Sigma\subset \R^{n+1}$. Using this fact, Theorem \ref{thm:main theorem type I singularity then entropy has lower bound} immediately implies entropy bounds on higher dimensional shrinkers $\Sigma\subset \R^{n+1}$ of the form $\Sigma =\Gamma\times \R^{n-1}$, where $\Gamma$ is a closed planar shrinker.

The key idea for proving Theorem \ref{thm:main theorem existence of perturbation reduce the entropy} is the notion of {\it entropy instability}. A shrinker $\Sigma\subset \R^{n+1}$ is {\it entropy unstable} if there exists a variation $(\Sigma_{\epsilon})_{\epsilon\in(-\delta,\delta)}$ of $\Sigma$ which decreases entropy. Entropy stability of embedded shrinkers was first introduced by Colding and Minicozzi in \cite{CoMi}. Our method to construct closed curves with turning number $m$ but entropy less than that of the $m$-covered circle $\Gamma_m$ is to show that $\Gamma_m$ is entropy unstable for all $m\geq 2$. We not only show that multiply-covered circles are entropy unstable, but also that {\it all} closed shrinkers $\Gamma\subset \R^2$ other than the embedded circle are entropy unstable. Furthermore, we calculate their {\it $F$-index}, which measures the number of linearly independent variations which reduce the entropy of the shrinker (see Liu \cite{Li}). 

Based on their work on the stability of self-shrinkers, Colding and Minicozzi defined a {\it piecewise MCF} starting at a closed hypersurface $\Sigma\subset \R^{n+1}$ as a finite collection of MCFs $\Sigma_t^i$ on time intervals $[t_i,t_{i+1}]$ so that each $\Sigma_{t_{i+1}}^{i+1}$ is a graph over $\Sigma_{t_{i+1}}^i$ of a function $u_{i+1}$ and
\begin{align}
\lambda(\Sigma_{t_{i+1}}^{i+1})&\leq \lambda(\Sigma_{t_{i+1}}^i). \label{eq:entropy condition 1}
\end{align}
Piecewise MCF provides an ad hoc notion of generic MCF that is important in many applications: it provides a method of continuing a flow through unstable singularities. If a MCF reaches an entropy unstable singularity, the flow is slightly perturbed to decrease entropy so that the singularity can never re-occur along the flow. Colding and Minicozzi proposed a piecewise MCF for closed embedded surfaces in $\R^3$ which becomes extinct in a round point \cite{CoMi}. We realize a similar result by constructing a piecewise CSF for closed immersed curves in $\R^2$.

\begin{namedtheorem}[\hypertarget{thm:PCSF}{Theorem B}]
Let $\Gamma\subset \R^2$ be a closed immersed curve. Then there exists a piecewise CSF starting at $\Gamma$ and defined up to a time $T>0$ at which the flow either becomes extinct in an embedded circle, or has type II singularities. Moreover, if $\Gamma$ has turning number greater than 1, the latter case holds.
\end{namedtheorem}

It is worth noting that the break-points of the piecewise CSF constructed in Theorem \bref{thm:PCSF}{B} can be made arbitrarily small in the $C^{\infty}$ norm. Furthermore, a result of Gage-Hamilton \cite{GaHa} and Grayson \cite{Gr} states that the CSF starting at a closed embedded plane curve becomes convex and eventually extincts in a ``round point''. See \cite{And-Bry} for an elegant direct proof. It follows that the piecewise CSF of Theorem \bref{thm:PCSF}{B} starting at a closed embedded plane curve reduces to the usual CSF whose only singularity is modelled by an embedded circle. 

A singularity of a MCF defined up to time $T$ is of {\it type I} if the curvature blows up no faster than $(T-t)^{-1/2}$; otherwise, the singularity is of {\it type II}. Type I singularities are important in the singularity analysis of MCF because Huisken \cite{Hu} showed that a blow up limit of a type I singularity is a smooth shrinker. More generally, the limit of a blowup sequence of a MCF at a fixed spacetime point is modeled on a possibly singular shrinker \cite{Hu},\cite{Il94}. The study of shrinkers is therefore central in the analysis of singularity formation of the MCF. Shrinkers can equivalently be defined as immersed hypersurfaces $\Sigma\subset \R^n$ satisfying the equation
\be\label{eq:shrinker equation}
H=\frac{1}{2}\langle x,\mathbf{n}\rangle.
\ee 
Straightfoward calculations show that the hyperplanes passing through the origin, the round spheres of radius $\sqrt{2n}$ centered at the origin, and the cylinders of the form $S^k\times \R^{n-k}\subset \R^{n+1}$ are all shrinkers. Much work has concerned the construction and classification of shrinkers \cite{An},\cite{Bre},\cite{KaKlMo}, and for the case $n=1$, a complete classification is known \cite{AbLa}. 

In \cite{St}, Stone computed the entropy of embedded spheres (and thus also of shrinking cylinders), implying that \be 2>\lambda(S^1)> \lambda(S^2)> \cdots\to\sqrt{2}.\ee In \cite{Bra}, Brakke proved that hyperplanes have least entropy among all shrinkers, and that there is a gap between the entropy of a plane and the entropy of the next lowest shrinker. Colding-Ilmanen-Minicozzi-White \cite{CIMW} proved that the sphere is the shrinker with lowest entropy among all closed embedded shrinkers. This result has been generalized by Bernstein-Wang \cite{BeWa}, Zhu \cite{Zh}, and Ketover-Zhou \cite{KeZh}: Bernstein-Wang proved that the sphere has lowest entropy among all closed embedded hypersurfaces in $\R^{n+1}$ where $2\leq n\leq 6$; Ketover-Zhou proved the same result by using an alternative min-max method; Zhu removed the dimension restriction.  Very recently, Hershkovits and White \cite{HeWh} proved a lower bound and rigidity theorem for shrinkers in higher codimension: a closed embedded shrinker $\Sigma\in \R^{n+1}$ with non-trivial $k$-th homology has entropy greater than or equal to the entropy of $S^k$, and equality holds if and only if $\Sigma =S^k\times \R^{n-k}$. 

Our results can be viewed as an extension of these results to the case of closed {\it immersed} curves, providing the first lower entropy bounds for immersed shrinkers. The main difference between the embedded case and the immersed case is that immersed MCFs lack many of the nice properties of embedded MCFs. For example, the avoidance principle implies that embeddedness is preserved under MCF. Also, the level set flow weak solutions are easier to handle in the embedded case. Fortunately, the immersed CSF has simpler geometry, allowing us to control the behavior of the immersed curve under CSF.

We highlight that our theorem detects long time behaviour of CSF from the information of the initial data. Although the monotonicity formula for entropy can be used to rule out high entropy shrinkers as singularity models of a mean curvature flow, it does not tell us whether the singularities are type I or type II. Our theorem provides a criterion to show that the type II singularities must occur from the initial data.

Stability of Abresch-Langer curves has been studied in a different context by Epstein-Weinstein \cite{EpWe}. Our results show that the Jacobi operator of an Abresch-Langer curve has the same number of negative eigenvalues as the linear stability operator studied by Epstein-Weinstein. Their analysis does not incorporate the action of the group generated by rigid motions and dilations.

\paragraph{Outline of the Proofs.}
We now outline the proof of Theorem \ref{thm:main theorem type I singularity then entropy has lower bound}. 

The key result we use to conclude the proof is due to Au \cite{Au}. Recall that Au proved the following theorem in \cite{Au}.

\begin{thm}[Main Theorem of \cite{Au}, part (a)]
Let $\Gamma_{m,n}$ be an Abresch-Langer curve with turning number $m$ and $\gamma_0=\Gamma+\epsilon\mathbf{n}$ a perturbation with $|\epsilon|$ small. Then if $\epsilon>0$, the curve shortening flow $\gamma_t$ starting from $\gamma_0$ is asymptotic to an $m$-covered circle.
\end{thm}

Given a closed curve $\Gamma$ with turning number $m$, we run the CSF starting from this curve. If the flow hits a type I singularity, then after rescaling it must be an Abresch-Langer curve. The continuity of entropy shows that the entropy of $\Gamma$ is bounded from below by the entropy of this Abresch-Langer curve. Using Au's result and the monotonicity property of entropy, we show that this Abresch-Langer curve has entropy bounded from below by the entropy of the $m$-covered circle. Thus we conclude the proof.

The main technical challenge is showing that the entropy functional is continuous under small perturbations near a closed shrinker $\Gamma$. Note that the entropy functional is not continuous in general: a sequence of rescalings at a point on a sphere converges to a hyperplane in the limit, however, a hyperplane has entropy 1, while a sphere has entropy at least $\sqrt{2}$. We prove that the entropy functional is continuous near a closed shrinker $\Gamma\subset \R^2$ in two steps. 

Firstly, Proposition \ref{prop:continuous when compact} is a general result giving a sufficient condition for the $F$-functional being continuous at $\Sigma$. More precisely, when $(x_0,t_0)$ lies in a compact subset $K\subset \R^{n+1}\times (0,\infty)$ and variations of $\Sigma$ lie in a bounded subset $\mathcal{B}\subset C^{2,\alpha}(\Sigma)$, the $F$-functional is continuous. 

The second step is to provide a quantitative version of \cite[Lemma 7.7]{CoMi} to show that we may in fact apply Proposition \ref{prop:continuous when compact} to closed shrinkers $\Gamma\subset \R^2$. To that end, Theorem \ref{thm:entropy achieved in compact set} shows that the entropy of the perturbed curve $\Gamma+f\mathbf{n}$ with $f\in\mathcal{B}$ is achieved in $B_R(0)\times [T_0,T_1]$ for some $R>0$ and some $0<T_0<T_1<\infty$. The main technical challenge is to show that $T_0>0$. We discuss the details in Section \ref{S:Properties of Entropy}.

\bigskip
Next, we give an outline of the proof of Theorem \ref{thm:main theorem existence of perturbation reduce the entropy}. Our method to construct closed curves with turning number $m$ but entropy less than that of the $m$-covered circle $\Gamma_m$ is to show that $\Gamma_m$ is entropy unstable for all $m\geq 2$. Recall that the {\it $F$-functional} is defined to be
\be 
F_{x_0,t_0}(\Sigma)=(4\pi t_0)^{-n/2}\int_{\Sigma}e^{-\frac{|x-x_0|^2}{4t_0}d\mu}.
\ee
The entropy $\lambda(\Sigma)$ is then the supremum over all translations $x_0\in \R^{n+1}$ and dilations $t_0>0$ of the $F$-functional. It follows from Huisken's monotonicity formula \cite{Hu} that shrinkers are precisely the critical points of the $F_{0,1}$-functional. The $F_{0,1}$-functional is called the {\it Gaussian area}. Therefore, shrinkers are precisely the minimal hypersurfaces of $\R^{n+1}$ equipped with the conformally changed metric $g_{ij}=e^{-\frac{|x|^2}{2n}}\delta_{ij}$. We show that the entropy index of a shrinker is bounded below by the difference between the Morse index\footnote{Recall that the Morse index of a minimal surface $\Sigma\subset (M,g)$ is defined as the number of negative eigenvalues of the Jacobi operator determined by $(M,g)$ and $\Sigma$ (see \cite[p.\ 41]{CoMi-MinSur}).} of the shrinker (considered as a minimal surface of $(\R^{n+1},e^{-\frac{|x|^2}{2n}}\delta_{ij})$) and the dimension of the space spanned by the mean curvature function and the component functions of the normal vector field of the shrinker. Consequently, we reduce the calculation of the entropy index of a shrinker to the calculation of the Morse index. For $1$-dimensional shrinkers in $\R^2$, the Jacobi operator reduces to a 1-dimensional Sturm-Liouville operator. Using the well-developed theory of spectra of Sturm-Liouville operators, we are able to determine exactly how many negative eigenvalues the Jacobi operator of a 1-dimensional closed shrinker has, and thus also the entropy index.

Lastly, we outline the proof of Theorem \bref{thm:PCSF}{B}. As mentioned above, the possible singularities of the CSF starting at $\Gamma$ are classified into two categories: type I and type II. Huisken \cite[Theorem 3.5]{Hu} shows that any rescaling of a type I singularity is a shrinker. By Abresch and Langer's classification of shrinkers for the CSF \cite{AbLa}, any rescaling of a type I singularity must therefore be an Abresch-Langer curve, or a (multiply-covered) circle. Corollary \ref{cor:stability of closed plane shrinkers} states that the only entropy stable closed singularity for CSF is the embedded circle.

Let $\Gamma$ be any closed immersed plane curve, which is the initial data for a CSF. If the singularity of the CSF is of type I, any blowup gives an Abresch-Langer curve or a circle, say $\tilde{\Gamma}_{\infty}$. If $\tilde{\Gamma}_{\infty}$ is not an embedded circle, Lemma \ref{lem:PCSF is possible} shows that the time slice of the flow can be slightly perturbed to a curve $\Gamma'$ with $\lambda(\Gamma')<\lambda(\tilde{\Gamma}_{\infty})$. Consequently, $\tilde{\Gamma}_{\infty}$ can never appear as a singularity for the flow starting from $\Gamma'$. The above process is then repeated with $\Gamma'$ instead of $\Gamma$. Since the perturbations can be made with arbitrarily small $C^{\infty}$-norm, the perturbations preserve turning number. Therefore piecewise CSF preserves turning number, and since there are only finitely many closed shrinkers of a given turning number, the piecewise flow terminates after finitely many perturbations.

\paragraph{Organization of the Paper.}
In Section \ref{S:Properties of Entropy} we prove some basic properties of the entropy functional for immersed hypersurfaces and shrinkers. In particular we prove the continuity Theorem \ref{thm:entropy achieved in compact set} for entropy, which is central to our argument.

In Section \ref{S:Entropy and Turning Number}, we prove Theorem \ref{thm:main theorem type I singularity then entropy has lower bound}.

In Section \ref{S:Entropy Index and Morse Index}, we introduce and study some basic properties of index of the shrinkers. The results of this section hold for shrinkers of any dimension.

In Section \ref{S:Entropy Stability of CSF Singularities}, we compute the precise entropy index of closed immersed shrinkers for the CSF. As a corollary, we obtain Theorem \ref{thm:main theorem existence of perturbation reduce the entropy}.

In Section \ref{S:Generic CSF} we combine our results to prove Theorem \bref{thm:PCSF}{B}.

\paragraph{Notation.}
Throughout, all hypersurfaces $\Sigma\subset \R^{n+1}$ are assumed to be immersed and orientable, so that there exists a globally defined unit normal vector field $\mathbf{n}:\Sigma\to S^n$. We denote by $x:\Sigma\to \R^{n+1}$ the given immersion and by $\Sigma+f\mathbf{n}$ the normal variation of $\Sigma$ by a function $f:\Sigma\to \R$. The mean curvature of $\Sigma$ is denoted by $H:\Sigma \to \R$. One can show (see \cite{Ma}, for example ) that MCF is a geometric flow which is invariant under tangential reparameterization, so we may also use the image $(\Sigma_t)_{t\in[0,T)}$ to denote the flow. In the case of curves, we use the notation $(\Gamma_t)_{t\in [0,T)}$. Finally, angle brackets $\langle\, \cdot\,,\, \cdot\, \rangle$ denote the standard Euclidean inner product. 

\paragraph{Acknowledgments.}
The authors thank Professor W. Minicozzi for suggesting this topic and for his support, as well as Professor D. Maulik, Professor A. Moitra and Dr. S. Gerovitch for supporting this project via the Summer Program in Undergraduate Research. The first author is indebted to C. Mantoulidis for introducing him to the mean curvature flow and for many interesting discussions. 

\section{Properties of Entropy}\label{S:Properties of Entropy}

\paragraph{}
In this section, we will prove that the entropy functional is continuous near a closed planar shrinker under certain conditions. In general, the entropy functional is not a continuous functional on the space of immersed hypersurfaces (with the $C^{\infty}$ topology, say). The standard example is given by blowing up a sphere at a point, giving a hyperplane in the limit. Each element of the blowup sequence is a sphere, thus having entropy $\sqrt{2}$. However, a hyperplane has entropy $1$. The results of this section will be used in the Section \ref{S:Entropy and Turning Number} to prove Theorem \ref{thm:main theorem type I singularity then entropy has lower bound}.

The following proposition shows that, when a closed hypersurface is slightly perturbed, the $F$-functional cannot change too much under the perturbation, as long as $(x_0,t_0)$ are confined to a compact subset of $\R^{n+1}\times (0,\infty)$. From now on, in this paper, we fix $\alpha\in (0,1)$ once and for all.

\begin{prop}\label{prop:continuous when compact}
Let $\Sigma\subset \R^{n+1}$ be a closed hypersurface and $K\subset \R^{n+1}\times (0,\infty)$ be a compact subset. Then there exist $\epsilon_0>0$ and a constant $C$, both depending only on $\Sigma$ and $K$, such that, for all $f\in C^{2,\alpha}(\Sigma)$ with $\|f\|_{C^{2,\alpha}}\leq 1$, all $\epsilon\geq 0$ with $\epsilon<\epsilon_0$, and all $(x_0,t_0)\in K$, the following inequality holds: \be \left|F_{x_0,t_0}(\Sigma_{\epsilon})-F_{x_0,t_0}(\Sigma)\right|\leq C\epsilon ,\ee where $\Sigma_{\epsilon}=\Sigma+ \epsilon f\mathbf{n}$.
\end{prop}

\begin{proof}
Let $x:\Sigma \to \R^{n+1}$ be the given immersion and let $f\in C^{2,\alpha}(\Sigma)$ with $\|f\|_{C^{2,\alpha}}\leq 1$. In local coordinates around a point $p\in \Sigma$, we can write
\begin{align*}
\tilde{g}_{ij}
&= \Big\langle \partial_i (x +\epsilon f\mathbf{n}),\, \partial_j (x +\epsilon f\mathbf{n})\Big\rangle \\
&=g_{ij}+2\epsilon fh_{ij}+\epsilon^2f^2h_{ik}h^k_{\,j}+\epsilon^2(\partial_i f)(\partial_j f),
\end{align*} where $g_{ij}$ is the metric induced on $\Sigma$ as a hypersurface in $\R^{n+1}$, and $\tilde{g}_{ij}$ is the metric induced on $\Sigma_\epsilon$. Since $\Sigma$ is closed, there exists $\epsilon_0>0$, such that, for any $f\in C^{2,\alpha}(\Sigma)$ with $\|f\|_{C^{2,\alpha}}\leq 1$ and any $\epsilon\in \R$ satisfing $|\epsilon|<\epsilon_0$, the tensor $\tilde{g}$ defines a metric on $\Sigma_{\epsilon}$ and the volume measure $\sqrt{\tilde{g}}$ and the derivative $\partial\tilde{g}/\partial \epsilon$ are uniformly bounded on compact sets. Define $$G_{x_0,t_0}(\beta,x)=e^{-\frac{|x-x_0+\beta f\mathbf{n}|^2}{4t_0}}\sqrt{\tilde{g}}.$$ Then Jacobi's formula gives
\begin{align*}
\frac{\partial G_{x_0,t_0}(\beta,x)}{\partial \beta}
&=\frac{1}{2}\left(\mathrm{Tr}\left(\tilde{g}^{-1}\frac{\partial \tilde{g}}{\partial \beta}\right)-\frac{f\langle x-x_0,\mathbf{n}\rangle+\beta f^2}{t_0}\right)e^{-\frac{|x-x_0+\beta f\mathbf{n}|^2}{4t_0}}\sqrt{\tilde{g}},
\end{align*}
which is uniformly bounded on $K$ by a constant depending only on $\epsilon_0$ and $K$. In particular, the constant is independent of the choice of $f\in C^{2,\alpha}(\Sigma)$ with $\|f\|_{C^{2,\alpha}}\leq 1$. Therefore
\begin{align*}
\left|F_{x_0,t_0}(\Sigma_{\epsilon})-F_{x_0,t_0}(\Sigma)\right| 
&=\left|\frac{1}{\sqrt{4\pi t_0}}\int_{\Sigma}(G_{x_0,t_0}(\epsilon,x)-G_{x_0,t_0}(0,x))dx\right|\\
&\leq \frac{1}{\sqrt{4\pi t_0}}\int_{\Sigma}\left|G_{x_0,t_0}(\epsilon,x)-G_{x_0,t_0}(0,x)\right|dx \\
&= \frac{\epsilon}{\sqrt{4\pi t_0}}\int_{\Sigma}\left|\int_0^1\frac{\partial G_{x_0,t_0}(\epsilon u,x)}{\partial \beta}du\right|dx \\
&\leq \frac{\epsilon}{\sqrt{4\pi t_0}}\int_{\Sigma}\int_0^1\left|\frac{\partial G_{x_0,t_0}(\epsilon u,x)}{\partial \beta}\right|du\, dx \\
&\leq \epsilon C
\end{align*}
as desired.
\end{proof}

\paragraph{}
In the rest of this section, we will focus on the case of curves in the plane. The main result of this section, below, shows that there exists a compact subset of $\R^2\times (0,\infty)$ such that the entropy of any curve obtained by perturbing a given closed plane shrinker is attained in this compact set. As a consequence, we may apply the previous proposition to conclude that the entropy functional is continuous under small perturbations of a closed plane shrinker. Theorem \ref{thm:entropy achieved in compact set} may be viewed as a quantitative version of \cite[Lemma 7.7]{CoMi} for immersed curves.

\begin{thm}\label{thm:entropy achieved in compact set}
Let $\Gamma\subset\R^{2}$ be a closed shrinker. Then there exist $\epsilon_0>0$ and a compact subset $K\subset \R^{2}\times (0,\infty)$, both depending only on $\Gamma$, such that, for any variation $f\in C^{2,\alpha}(\Sigma)$ with $\|f\|_{C^{2,\alpha}}\leq 1$ and any $\epsilon \in \R$ with $|\epsilon|<\epsilon_0$, the following holds: \[\lambda(\Gamma_{\epsilon})=\sup_{(x_0,t_0)\in K}F_{x_0,t_0}(\Gamma_{\epsilon}),\] where $\Gamma_{\epsilon}=\Gamma+\epsilon f\mathbf{n}$.
\end{thm}

Before proving Theorem \ref{thm:entropy achieved in compact set}, let us recall an application of Huisken's monotonicity to self-shrinkers. Suppose $\Sigma$ is a self-shrinker in $\R^{n+1}$ (not necessarily embedded), then for any $y\in \R^{n+1}$ and $a\in \R$, the function 
\[g(s)=F_{sy,1+as^2}(\Sigma)\]
satisfies $g'(s)\leq 0$ for all $s>0$ with $1+as^2>0$. See \cite[Eq.\ 7.13]{CoMi}. Moreover, $g'(s)=0$ if and only if $as x^\bot+y^\bot=0$ for any $x\in\Sigma$, see \cite[Eq.\ 7.27]{CoMi}. As a corollary of this monotonicity formula, we have the following Lemma, see section 7.2 of \cite{CoMi}. Again, $\Sigma$ is not necessarily embedded, because the monotonicity formula also holds for an immersed mean curvature flow.

\begin{lemma}\label{Lem:entropy of a self-shrinker is achieved at F01}
Then entropy of a self-shrinker $\Sigma$ is only achieved at $F_{0,1}(\Sigma)$ if $\Sigma$ does not split along a line.
\end{lemma}

Now we are ready to prove Theorem \ref{thm:entropy achieved in compact set}.
\begin{proof}[Proof of Theorem \ref{thm:entropy achieved in compact set}]
Throughout, we will denote by $\Gamma^f_{\epsilon}$ the perturbed curve $\Gamma_{\epsilon}^f=\Gamma+\epsilon f\mathbf{n}$ if $f\in C^{2,\alpha}(\Gamma)$. 
As stated in the proof of the previous proposition, there exists an $\epsilon_0>0$ such that $\Gamma_\epsilon^f$ is a closed immersed curve for any $f\in C^{2,\alpha}(\Gamma)$ with $\|f\|_{C^{2,\alpha}}\leq 1$ and any $\epsilon\in \R$ with $|\epsilon|<\epsilon_0$. 
Let $\mathcal{B}$ be the set consisting of those $\Gamma_{\epsilon}^f$ with $f\in C^{2,\alpha}(\Sigma)$ such that $\|f\|_{C^{2,\alpha}}\leq 1$ and $\epsilon \in \R$ such that $|\epsilon|<\epsilon_0$. 

{\it Step (1): Bound on $|x_0|$.} By the compactness of $\Gamma$ and the uniform $C^{2,\alpha}$ bound of the variational functions, there exists $R>0$ (independent of $f$ and $\epsilon$) such that $\Gamma_{\epsilon}^f$ is contained in the closed ball $B_{R}(0)\subset \R^{2}$ of radius $R$ centered at the origin, which is compact.
Colding and Minicozzi \cite{CoMi} showed that for any closed hypersurface $M\subset \R^{n+1}$ and each fixed $t_0>0$, $\sup_{x_0\in \R^{n+1}}F_{x_0,t_0}(M)$ is achieved inside the convex hull of $M$. The convex hull of $M$ is compact since $M$ is closed. Therefore, for each fixed $t_0$ and all $\Gamma_{\epsilon}^f\in \mathcal{B}$, $\sup_{x_0\in \R^{2}}F_{x_0,t_0}(\Gamma_{\epsilon}^f)$ is attained in $B_{R}(0)\subset \R^{2}$.

{\it Step (2): Upper bound on $t_0$.} Let $L$ be the length of $\Gamma$. When $\epsilon_0$ is sufficiently small, the length of $\Gamma_\epsilon^f\in\mathcal{B}$ is bounded by $2L$. Let $T_1=(4\pi)^{-1}(3L)^{2}$. Then for any $t_0>T_1$, any $x_0\in B_R(0)$, and any curve $\Gamma_{\epsilon}^f\in \mathcal{B}$, we have 
\begin{align*}
F_{x_0,t_0}(\Gamma_{\epsilon}^f)
&\leq (4\pi t_0)^{-1/2}\mathrm{Length}(\Gamma_{\epsilon}^f) < (4\pi T_1)^{-1/2}2L = 2/3<1.
\end{align*}
However, we know that $\lambda(\Gamma_{\epsilon}^f)\geq 1$. This can be seen by zooming in at any point on the curve. We conclude that \[ \sup_{(x_0,t_0)\in B_{R}(0)\times (0,\infty)}F_{x_0,t_0}(\Gamma_{\epsilon}^f) =\sup_{(x_0,t_0)\in B_{R}(0)\times (0,T_1]}F_{x_0,t_0}(\Gamma_{\epsilon}^f), \] for all $\Gamma_{\epsilon}^f\in \mathcal{B}$. 

{\it Step (3): Lower bound on $t_0$.} We observe that for $\epsilon_0$ small enough, $\Gamma^f_\epsilon$ has a uniform curvature bound. This follows from the expression of curvature in terms of the first and second derivatives of the immersion. Therefore, we may assume that $|k|\leq k_0$ for all $\epsilon\leq \epsilon_0$.  Also, if $\epsilon_0$ is small, then the multiplicity of each self-intersection points of $\Gamma_\epsilon^f$ are the same as the multiplicity of the self-intersection points of $\Gamma$. Let $R_0>0$ be a constant whose value will be determined. Now for any fixed $x_0$ and any $t_0\leq R_0^{-2}$, we have
\begin{align*}
F_{x_0,t_0}(\Gamma^f_\epsilon)=(4\pi t_0)^{-1/2}\int_{\{|x-x_0|\geq R_0^{-1}\}\cap \Gamma^f_\epsilon}&e^{-\frac{|x-x_0|^2}{4t_0}}ds \\
&+(4\pi t_0)^{-1/2}\int_{\{|x-x_0|\leq R_0^{-1}\}\cap \Gamma^f_\epsilon}e^{-\frac{|x-x_0|^2}{4t_0}}ds.
\end{align*}
The first term on the right hand side is bounded by
\[(4\pi t_0)^{-1/2}e^{-\frac{R_0^{-2}}{4t_0}}\mathrm{Length}(\Gamma_{\epsilon}^f)\leq CR_0^{-1/2}\]for a universal constant $C$. 

Now we estimate the second term on the right hand side. By change of variables, this term equals
\[\int_{\{|x-x_0|\leq R_0^{-1}t_0^{-1/2}\}\cap (t_0^{-1/2}\Gamma^f_\epsilon)}e^{-\frac{|x-x_0|^2}{4}}ds,\]
where $t_0^{-1/2}\Gamma^f_\epsilon$ is the curve $\Gamma^f_\epsilon$ rescaled by $t_0^{-1/2}$ with center at $x_0$. Suppose $\{|x-x_0|\leq R_0^{-1}t_0^{-1/2}\}\cap (t_0^{-1/2}\Gamma^f_\epsilon)$ consists of the connected components $\gamma_1,\dots,\gamma_m$. Then each $\gamma_i$ is a curve with curvature uniformly bounded by $k_0t_0$. In particular, for a given $\delta>0$, when $R_0$ is large enough depending on $k_0$, the curvature of $\gamma_i$ is small enough and therefore
\[\int_{\gamma_i}e^{-\frac{|x-x_0|^2}{4}}ds\leq 1+\delta.\]
This can be seem by a blow-up argument. In fact as $R_0\to\infty$, $t_0\to 0$, the uniform curvature bound implies that all $\gamma_i$ would converge to a straight line, and the Gaussian integral would converge to $1$.

 The number $m$ is bounded by the multiplicity of the self-intersection points of $\gamma_i$ when $\epsilon_0$ is sufficiently small and $R_0$ sufficiently large. So we conclude that 
\[F_{x_0,t_0}(\Gamma_{\epsilon}^f)\leq CR_0^{-1/2}+m(1+\delta)=m+\delta m+CR_0^{-1/2}.\]
Lemma \ref{Lem:entropy of a self-shrinker is achieved at F01} implies that the entropy of $\Gamma$ is achieved at $F_{0,1}(\Gamma)$, rather than blowing up a self-intersection point. That is, $\lambda(\Gamma)>m+\beta$ for some $\beta>0$. Therefore when $\epsilon_0$ is small, Proposition \ref{prop:continuous when compact} gives that $F_{0,1}(\Gamma^f_\epsilon)>m+\beta/2$. If we pick $\delta$ small enough and $R_0$ large enough, for $t_0\leq R_0^{-3}=:T_0$ we have
\[F_{x_0,t_0}(\Gamma_\epsilon^f)\leq m+\beta/4<m+\beta/2<\lambda(\Gamma).\]
Combining all of the steps above, the proposition follows, with $K:=B_R(0)\times [T_0,T_1]$. 
\end{proof}

 \begin{remark}
We believe that the self-intersection points of an Abresch-Langer curve actually have multiplicity $2$, i.e. an Abresch-Langer curve passes a point in the plane at most twice. It should follow from the fact that the Abresch-Langer curves are all convex.
\end{remark}

\begin{cor}\label{cor:limit of entropy is entropy of limit 1}
Let $\Gamma\subset\R^{2}$ be a closed shrinker and let $(\Gamma_i)_{i\in \N}$ be a sequence of plane curves which $C^\infty$ converges to $\Gamma$. Then 
\[\lim_{i\to \infty}\lambda(\Gamma_i)=\lambda(\Gamma).\]
\end{cor}

\begin{proof}
It follows from the convergence assumption that there exist an $N'\in \N$, such that, for $i>N'$, $\Gamma_i$ can be written as a graph \[\Gamma_i=\Gamma+\epsilon_i f_i \mathbf{n},\] for some $f_i\in C^{\infty}(\Gamma)$ normalized to $\| f_i\|_{C^\infty}= 1$ and some $\epsilon_i>0$ with $\epsilon_i\to 0$ as $i\to \infty$. Applying the previous theorem implies that there exists an $N''\geq N'$ and a compact subset $K\subset \R^2\times (0,\infty)$ such that, for all $i>N''$, the entropy $\lambda(\Gamma_i)$ is achieved at some $(x_i,t_i)\in K$. Then by Proposition \ref{prop:continuous when compact}, there exist an $N\geq N''$ and a constant $C$ depending only on $K$ such that, for any $(x_0,t_0)\in K$ and any $i>N$, we have 
\[F_{x_0,t_0}(\Gamma_i)-C\epsilon_i\leq F_{x_0,t_0}(\Gamma)\leq F_{x_0,t_0}(\Gamma_i)+C\epsilon_i.\]
Consequently, \[\sup_{K}F_{x_0,t_0}(\Gamma_i)-C\epsilon_i\leq \sup_{K}F_{x_0,t_0}(\Gamma)\leq \sup_{K} F_{x_0,t_0}(\Gamma_i)+C\epsilon_i.\] Since the entropy $\lambda(\Gamma_i)$ is achieved in $K$ for all $i>N$, it follows immediately that 
\[\lambda(\Gamma_i)-C\epsilon_i\leq \lambda(\Gamma)\leq \lambda(\Gamma_i)+C\epsilon_i.\] Now taking the limit $i\to \infty$ gives the desired result.
\end{proof}

\paragraph{}
The following corollary will be central to our construction of a generic CSF. The previous corollary shows that the entropy is continuous {\it at} a closed plane shrinker in the space of immersed curves. However, for the construction of the piecewise CSF in Theorem \bref{thm:PCSF}{B}, we also need the entropy to be continuous {\it near} a closed plane shrinker in the space of immersed curves. The following corollary ensures that the latter statement is satisfied.

\begin{cor}\label{cor:continuity of entropy near a shrinker}
Let $\Gamma\subset \R^2$ be a closed shrinker, $f\in C^{2,\alpha}(\Gamma)$ be a variation, and $(g_i)_{i\in \N}\subset C^{\infty}(\Gamma)$ be a sequence of functions which $C^{\infty}$-converges to $0$. Then there exists an $\epsilon_0>0$ such that, for all $\epsilon\in \R$ with $|\epsilon|<\epsilon_0$, we have \[\lim_{i\to \infty}\lambda(\Gamma_{\epsilon}+g_i\mathbf{n})=\lambda(\Gamma_{\epsilon}),\] where $\Gamma_{\epsilon}=\Gamma+\epsilon f\mathbf{n}$.
\end{cor}

\begin{proof}
By the $C^{\infty}$ convergence assumption, there exists $N>0$ such that $\|g_i\|_{C^{2,\alpha}}\leq 1/2$ for all $i>N$. Since $\Gamma$ is compact and $f\in C^{2,\alpha}(\Gamma)$, there exists $\epsilon_1>0$ such that $\|\epsilon f\|_{C^{2,\alpha}}\leq 1/2$ for all $\epsilon \in \R$ with $|\epsilon| <\epsilon_1$. Consequently, $\|\epsilon f+g_i\|_{C^{2,\alpha}}\leq 1$ for all $i>N$ and all $\epsilon \in \R$ with $|\epsilon|<\epsilon_1$. By Theorem \ref{thm:entropy achieved in compact set}, there exists an $\epsilon_2>0$ with $\epsilon_2\leq \epsilon_1$, an $N'\geq N$, and a compact subset $K\subset \R^2\times (0,\infty)$ such that, for all $\epsilon \in \R$ with $|\epsilon|<\epsilon_2$ and all $i>N'$, the entropy of the curves $\Gamma_{\epsilon}$ and $\Gamma_{\epsilon}+g_i\mathbf{n}$ is attained in $K$. By Proposition \ref{prop:continuous when compact}, for all $\epsilon\in \R$ with $|\epsilon|<\epsilon_0$ and all $(x_0,t_0)\in K$, we have \[\lim_{i\to \infty}F_{x_0,t_0}(\Gamma_{\epsilon}+g_i\mathbf{n})=F_{x_0,t_0}(\Gamma_{\epsilon}).\] Since the curves $\Gamma_{\epsilon}$ and $\Gamma_{\epsilon}+g_i\mathbf{n}$ attain their entropies in $K$ for all $\epsilon \in \R$ with $|\epsilon|<\epsilon_0$ and all $i>N'$, arguing as in the proof of the previous corollary concludes the proof.
\end{proof}

\section{Entropy and Turning Number}\label{S:Entropy and Turning Number}

\paragraph{}
In this section we will prove Theorem \ref{thm:main theorem type I singularity then entropy has lower bound}, which gives lower bounds for the entropy of closed plane curves which generate type I singularities under the CSF. We will first prove a particular case, in which we only consider the closed plane shrinkers themselves.

\begin{thm}\label{thm:main theorem A-L curve lower bound}
Let $\Gamma\subset \R^2$ be a closed shrinker with turning number $m$. Then $\lambda(\Gamma)\geq\lambda(\Gamma_m).$
\end{thm}

The proof is based on a result by Au \cite{Au} together with the idea from \cite{CIMW}. Let us first recall the Theorem proved by Au.

\begin{thm}[Main Theorem of \cite{Au}, part (a)]
Let $\Gamma_{m,n}$ be an Abresch-Langer curve with turning number $m$ and $\gamma_0=\Gamma+\epsilon\mathbf{n}$ be a perturbation with $|\epsilon|$ small. Then if $\epsilon>0$, the curve shortening flow $\gamma_t$ starting from $\gamma_0$ is asymptotic to an $m$-covered circle.
\end{thm}

Now we give the proof of Theorem \ref{thm:main theorem A-L curve lower bound}.

\begin{proof}[Proof of Theorem \ref{thm:main theorem A-L curve lower bound}]
We assume $\Gamma\neq \Gamma_m$, since otherwise the result holds trivially. Let $x:S^1\to \R^2$ be the given immersion of $\Gamma$. There are two cases to consider.

{\it Case (1):} Suppose that the degree of the map $x:S^1\to x(S^1)$ is 1 (in other words, $\Gamma$ is not multiply-covered). Define the constant normal variation function $f=1$ along $\Gamma$. By Corollary \ref{cor:limit of entropy is entropy of limit 1}, if $\Gamma_{\epsilon}=\Gamma+\epsilon \mathbf{n}$, then \[\lim_{\epsilon\to 0}\lambda(\Gamma_{\epsilon})=\lambda(\Gamma).\] In particular, there exists a constant $C>0$ such that \[\lambda(\Gamma)-C\epsilon\leq\lambda(\Gamma_{\epsilon})\leq \lambda(\Gamma)+C\epsilon,\] for all $\epsilon>0$ small enough. Au \cite{Au} showed that the rescaled CSF starting at $\Gamma_{\epsilon}$ converges to the the $m$-covered circle. By the monotonicity property of entropy, we therefore have \[\lambda(\Gamma_m)\leq \lambda(\Gamma_{\epsilon})\leq \lambda(\Gamma)+C\epsilon.\] Taking the limit $\epsilon \to 0$ shows that $\lambda(\Gamma)\geq \lambda(\Gamma_m)$. 

{\it Case (2):} Suppose that the degree of the map $x:S^1\to x(S^1)$ is greater than 1 (in other words, $\Gamma$ is a multiply-covered Abresch-Langer curve). Then $\Gamma$ is a $k$-covered Abresch-Langer curve (with $k\geq 2$) $\Gamma'$, where $\Gamma'$ has turning number $p$ satisfying $kp=m$. By the previous case and (\ref{eq:entropy of m-covered circle}), \[ \lambda(\Gamma)=k\cdot\lambda(\Gamma')\geq k\cdot\lambda(\Gamma_p)=kp\cdot\lambda(\Gamma_1)=m\cdot\lambda(\Gamma_1)=\lambda(\Gamma_m).\]
\end{proof}

\begin{lemma}
The turning number of a closed immersed curve $\Gamma\subset \R^2$ is preserved under the CSF and rescaled CSF.
\end{lemma}

\begin{proof}
CSF and rescaled CSF preserve immersedness and are therefore regular homotopies. Regular homotopy classes of immersions $S^1\to \R^2$ are classified by their turning number, by the Whitney-Graustein theorem \cite[Theorem 1]{Wh}. In other words, two closed immersed planar curves are regularly homotopic if and only if they have the same turning number.
\end{proof}

\begin{proof}[Proof of Theorem \ref{thm:main theorem type I singularity then entropy has lower bound}]
We assume that $\Gamma$ is not a shrinker, since otherwise the theorem follows immediately from Theorem \ref{thm:main theorem A-L curve lower bound}. Let $(\Gamma_t)_{t\in [0,T)}$ be the CSF starting at $\Gamma$. Since all type I singularities are closed shrinkers, they are compact, so there exists a unique singular point $x_0\in \R^2$. Without loss of generality, we may assume $x_0$ is the origin. Let $(\tilde{\Gamma}_{\tau})_{\tau\in [\log T^{-1/2},\infty)}$ be a rescaled CSF \cite[Section 2]{Hu} around $0\in \R^2$ starting at $\Gamma$. By \cite[Theorem 3.5]{Hu}, for any sequence $\tau_i\to \infty$, there exists a subsequence also denoted $\tau_i$ such that the rescaled curves $\tilde{\Gamma}_{\tau_i}$ will smoothly converge to a shrinker $\tilde{\Gamma}_{\infty}$ as $\tau_i\to \infty$. By the classification of $1$-dimensional shrinkers \cite{AbLa}, $\tilde{\Gamma}_{\infty}$ must be a multiply-covered circle or a multiply-covered Abresch-Langer curve. 

Let $m$ be the turning number of $\tilde{\Gamma}_{\infty}$. By the lemma, the turning number of $\tilde{\Gamma}_{\tau}$ is $m$ for all $\tau\in [\log T^{-1/2},\infty)$. By the $C^{\infty}$-convergence $\tilde{\Gamma}_{\tau_i}\to\tilde{\Gamma}_{\infty}$, the curvature of $\tilde{\Gamma}_{\tau_i}$ smoothly converges to the curvature of $\tilde{\Gamma}_{\infty}$. Therefore the total curvatures satisfy \[n=\lim_{i\to \infty}\frac{1}{2\pi}\int_{\tilde{\Gamma}_{\tau_i}}\tilde{k}_i=m.\] This shows that $\tilde{\Gamma}_{\infty}$ is a shrinker of turning number $m$. By the monotonicity of entropy, \[\lambda(\tilde{\Gamma}_{\tau})\leq \lambda(\Gamma)\] for all $\tau\in[\log T^{-1/2},\infty)$. By the $C^{\infty}$-convergence $\tilde{\Gamma}_{\tau_i}\to\tilde{\Gamma}_{\infty}$, Corollary \ref{cor:limit of entropy is entropy of limit 1} implies that \[\lim_{i\to \infty}\lambda(\tilde{\Gamma}_{\tau_i})=\lambda(\tilde{\Gamma}_{\infty}),\]
and since this holds for any blowup sequence, it follows thus that
\[\lambda(\Gamma)\geq\lambda(\tilde{\Gamma}_{\infty}).\]
Since we assume that $\Gamma$ is not a shrinker and since entropy is a constant along the CSF only for shrinkers, the inequality is in fact strict. The theorem now follows from Theorem \ref{thm:main theorem A-L curve lower bound}.
\end{proof}

\section{Entropy Index, $F$-Index and Morse Index}\label{S:Entropy Index and Morse Index}

\paragraph{}
Recall that the shrinkers $\Sigma\subset \R^{n+1}$ are precisely the minimal surfaces of $\R^{n+1}$ equipped with the conformally changed metric $g_{ij}=e^{-\frac{|x|^2}{2n}}\delta_{ij}$, and are critical points of the Gaussian area functional $F=F_{0,1}$. In this section, we will determine the relationship between three types of stability of a shrinker: entropy stability, $F$-stability, and Morse stability (when a shrinker is considered as a critical point of $F=F_{0,1}$). Theorem \ref{thm:equivalence of F and Morse index}, shows that a closed shrinker $\Sigma\subset \R^{n+1}$ is entropy unstable if the Morse index of $\Sigma$ is greater than $n+2$. Consequently, to prove Theorem \ref{thm:main theorem existence of perturbation reduce the entropy} we will only need to show that the $m$-covered circle has Morse index greater than 3. The results of this section apply to closed shrinkers $\Sigma^n\subset \R^{n+1}$ of any dimension.

Entropy is hard to compute. Therefore it is not easy to characterize the stability of entropy by variational methods. In order to overcome this issue, Colding-Minicozzi \cite{CoMi} introduced the notion of $F$-stability, and study the connections between $F$-stability and entropy stability. Later Liu \cite{Li} studied further properties of these connections. Let us recall some definitions in \cite{CoMi} and \cite{Li}. For the purpose of this paper, we will only discuss the case of closed shrinkers.

\begin{defn}
[Stability, see \cite{CoMi}]
Let $\Sigma^n\subset \R^{n+1}$ be a closed shrinker and $\Sigma_{\epsilon}$ a variation of $\Sigma$ in the direction $f\in C^{2,\alpha}(\Sigma)$.
\begin{description}
\item[(i)]
The variation $\Sigma_{\epsilon}$ is {\it entropy unstable} if $\lambda(\Sigma_\epsilon)< \lambda(\Sigma)$ for all sufficiently small $\epsilon\neq 0$; otherwise, the variation is {\it entropy stable}.
\item[(ii)]
The variation $\Sigma_{\epsilon}$ is {\it $F$-stable} if there exist some variations $x_\epsilon$ of $0$ and $t_\epsilon$ of $1$ such that $$ \frac{\partial^2}{\partial\epsilon^2}\Big|_{\epsilon=0}F_{x_\epsilon,t_\epsilon}(\Sigma_\epsilon)\geq 0;$$ otherwise, the variation is {\it $F$-unstable}.
\item[(iii)]
The variation $\Sigma_{\epsilon}$ is {\it Morse stable} if 
$$
\frac{\partial^2}{\partial \epsilon^2}\Big|_{\epsilon=0}F(\Sigma_{\epsilon})\geq 0;
$$
otherwise, the variation is {\it Morse unstable}.
\end{description}
\end{defn}

\begin{defn}
[Index, see \cite{Li}]
Let $\Sigma^n\subset \R^{n+1}$ be a closed shrinker. 
\begin{description}
\item[(i)]
The {\it entropy index} of $\Sigma$, denoted by $\ind_{\lambda}(\Sigma)$, is the maximum dimension of all subspaces $V\subset C^{2,\alpha}(\Sigma)$ such that, for any nonzero $g\in V$, the variation $\Sigma_{\epsilon}$ in the direction of $g$ is entropy unstable. The shrinker $\Sigma$ is {\it entropy stable} if its entropy index is zero; otherwise, $\Sigma$ is {\it entropy unstable}.
\item[(ii)]
The {\it $F$-index} of $\Sigma$, denoted by $\ind_{F}(\Sigma)$, is the maximum dimension of all subspaces $V\subset C^{2,\alpha}(\Sigma)$ such that, for any nonzero $g\in V$, the variation $\Sigma_{\epsilon}$ in the direction of $g$ is $F$-unstable. The shrinker $\Sigma$ is {\it $F$-stable} if its $F$-index is zero; otherwise, $\Sigma$ is {\it $F$-unstable}.
\item[(iii)]
The {\it Morse index} of $\Sigma$, denoted by $\ind_{M}(\Sigma)$, is the maximum dimension of all subspaces $V\subset C^{2,\alpha}(\Sigma)$ such that, for any nonzero $g\in V$, the variation $\Sigma_{\epsilon}$ in the direction of $g$ is Morse unstable. The shrinker $\Sigma$ is {\it Morse stable} if its Morse index is zero; otherwise, $\Sigma$ is {\it Morse unstable}.
\end{description}
\end{defn}

\begin{remark}
In the definition of the indexes, the variations are assumed to be $C^{2,\alpha}$. However, from standard elliptic theory (also see the discussion after (\ref{Eq:L})), one can define the indexes as the dimension of unstable variations with other regularities, such as $W^{1,2}$, and the indexes would be the same.
\end{remark}
\paragraph{}
In \cite[Theorem 0.15]{CoMi}, Colding-Minicozzi showed that an $F$-unstable self-shrinker must be entropy unstable, which provides the link between entropy and the $F$-functional. Their theorem was stated for embedded shrinkers, but their proof also works for closed immersed shrinkers. 

\begin{thm}[\cite{CoMi}]\label{thm:F-unstable implies entropy unstable}
Let $\Sigma^n\subset \R^{n+1}$ be a closed shrinker. Then the entropy index of $\Sigma$ is bounded below by the $F$-index of $\Sigma$, i.e.
\be
\ind_{\lambda}(\Sigma)\geq \ind_{F}(\Sigma).
\ee
\end{thm}

This theorem is just a re-formulation of Colding-Minicozzi's \cite[Theorem 0.15]{CoMi}. For a proof, we refer the readers to the proof of \cite[Theorem 0.15]{CoMi} (see \cite[page 789]{CoMi}).

\paragraph{}
In \cite{CoMi}, Colding-Minicozzi introduced the elliptic operator $L$, defined by
\begin{equation}\label{Eq:L}
L=\Delta+|A|^2-\frac{1}{2}\langle x,\nabla(\cdot)\rangle+\frac{1}{2},
\end{equation}
see \cite[(4.13)]{CoMi}. This operator is the linearized operator of a self-shrinker, and it is also the stability operator for the Gaussian area functional. Therefore, the Morse index of a self-shrinker can be also defined to be the number (count multiplicity) of negative eigenvalues of the operator $L$.

We point out that there are two eigenfunctions to the operator $L$ on a self-shrinker. It was proved by Colding-Minicozzi in \cite[Lemma 5.5]{CoMi} that
\begin{equation}\label{Eq:L eigenfunction}
LH=H,\qquad L\langle \mathbf{n},v\rangle=	\frac{1}{2}\langle \mathbf{n},v\rangle,
\end{equation}
where $v$ is a constant vector field in $\R^{n+1}$.

\paragraph{}
We will use two lemmas in the proof of Theorem \ref{thm:equivalence of F and Morse index}:

\begin{lemma}\label{lem:curvature and components of normal are stable}
Let $\Sigma\subset \R^{n+1}$ be a closed shrinker and let $f\in C^{2,\alpha}(\Sigma)$ be contained in the subspace of variations spanned by the mean curvature $H$ and the components $\mathbf{n}_1,\dots,\mathbf{n}_{n+1}$ of the normal vector field $\mathbf{n}:\Sigma\to \R$. Then $f$ is an $F$-stable variation.
\end{lemma}

\begin{proof}
We refer the reader to \cite[Section 4]{CoMi} for properties and notations related to the second variation of the $F$-functional. Let $f=aH+\langle Y,\mathbf{n}\rangle$ for some $a\in \R$ and some $Y\in \R^{n+1}$. Choose any $y\in \R^{n+1}$ and any $h\in \R$ and set $x_{\epsilon}=\epsilon y$, $t_{\epsilon}=1+\epsilon h$ and $\Sigma_{\epsilon}=\Sigma+\epsilon f\mathbf{n}$. Then  
\begin{align*}
\left.\frac{\partial^2}{\partial \epsilon^2}\right|_{\epsilon=0}F_{x_{\epsilon},t_{\epsilon}}(\Sigma_{\epsilon})
& =\int_{\Sigma}\left(-fLf+2hfH-h^2H^2+f\langle y,\mathbf{n}\rangle-\frac{1}{2}\langle y,\mathbf{n}\rangle^2\right)e^{-\frac{|x|^2}{4}}d\mu \\
& =\int_{\Sigma}\left(-a^2H^2-\frac{1}{2}\langle Y,\mathbf{n}\rangle^2+2ahH^2-h^2H^2+\langle Y\mathbf{n},\rangle\langle y,\mathbf{n}\rangle-\frac{1}{2}\langle y,\mathbf{n}\rangle^2\right)e^{-\frac{|x|^2}{4}}d\mu \\
& =-\int_{\Sigma}\left((a-h)^2H^2+\frac{1}{2}\left(\langle Y,\mathbf{n}\rangle-\langle y,\mathbf{n}\rangle\right)^2\right)e^{-\frac{|x|^2}{4}}d\mu.
\end{align*}
If we choose $h=a$ and $y=Y$, then $\left.\frac{\partial^2}{\partial \epsilon^2}\right|_{\epsilon=0}F_{x_{\epsilon},t_{\epsilon}}(\Sigma_{\epsilon})=0$. Therefore $f$ is an $F$-stable variation.
\end{proof}

In the following, we will suppose $x_i$'s are the standard coordinate functions in $\R^{n+1}$. We will use $\partial_{x_i}$ to denote the vector field generated by the coordinate function $x_i$.

\begin{lemma}\label{lem:surjectivity of Gauss map}
Let $\Sigma\subset \R^{n+1}$ be a closed, orientable hypersurface. Then the component functions $\mathbf{n}_i=\langle \mathbf n,\partial_{x_i}\rangle : \Sigma\to \R$ are linearly independent.
\end{lemma}

\begin{proof}
We claim that if the Gauss map $G:\Sigma\to S^n$ is surjective, then the lemma holds. Suppose that the Gauss map is surjective, and suppose for contradiction that the $\mathbf{n}_i$ are linearly dependent, i.e. there exists $y\in \R^{n+1}$ such that $\langle y,\mathbf{n}(p)\rangle=0$ for all $p\in \Sigma$. Then since the Gauss map is surjective, $y$ is orthogonal to every vector in $\R^{n+1}$, which is true if and only if $y$ is the zero vector.

Now we show that the Gauss map of any closed orientable hypersurface $x:M^n\to\R^{n+1}$ is surjective. Pick any vector $v\in S^n$. We wish to show that there exists some $p\in M$ such that $\mathbf{n}(p)=v$. By the compactness of $M$, the smooth function $\langle v,x\rangle :M \to \R$ attains a maximum at some point $p\in M$. Let $\{e_i\}$ be a basis of $T_pM$. Then \[0=\nabla_{e_i}\langle v,x(p)\rangle =\langle v,e_i\rangle\] for all $i=1,\dots,n$. Consequently, $v$ is orthogonal to $M$ at $p$, so $v=\pm\mathbf{n}(p)$. Since we have chosen $p$ to be the maximum, and since $\langle v,x\rangle$ is not constant ($M$ is closed), we must in fact have $v=\mathbf{n}(p)$, as desired.
\end{proof}

\begin{cor}\label{cor:curvature and components of normal are linearly independent}
Let $\Sigma\subset \R^{n+1}$ be a closed shrinker. Then \be \dim \mathrm{span}\{H,\mathbf{n}_1,\dots,\mathbf{n}_{n+1}\}=n+2,\ee where $\mathbf{n}_i=\langle \mathbf n,\partial_{x_i}\rangle$.
\end{cor}

\begin{proof}
Colding-Minicozzi have shown that, for all $i$, the functions $\mathbf{n}_i$ and $H$ are eigenfunctions of the stability operator $L$ determined by $\Sigma$, see \cite[Lemma 5.5]{CoMi} and (\ref{Eq:L eigenfunction}). Moreover, $\mathbf{n}_i$ and $H$ have distinct eigenvalues. It follows that the mean curvature $H$ is linearly independent of $\mathbf{n}_i$, for any $i$. The proposition now follows from Lemma \ref{lem:surjectivity of Gauss map}.
\end{proof}

\paragraph{}
We are now able to prove the main result of this section, which will be used in the next section to calculate the entropy index of Abresch-Langer curves and $m$-covered circles.

\begin{thm}\label{thm:equivalence of F and Morse index}
Let $\Sigma\subset \R^{n+1}$ be a closed shrinker. Then the $F$-index and Morse index are related by
\be \label{eq:F-index vs Morse index}
\ind_F(\Sigma) = \ind_M(\Sigma)-n-2.
\ee
Consequently, if 
\be
\ind_M(\Sigma)>n+2,
\ee
then $\Sigma$ is entropy unstable.
\end{thm}

\begin{proof}
By standard elliptic theory, the Morse index of an elliptic integrand is finite. In particular, $\ind_M(\Sigma)$ is finite. In what follows, $m=\ind_M(\Sigma)$.

First we show that $\ind_{F}(\Sigma)\leq m-n-2$. Combining with Lemma \ref{lem:curvature and components of normal are stable} and Corollary \ref{cor:curvature and components of normal are linearly independent} shows that it suffices to show that any $F$-unstable variation is a Morse unstable variation. Suppose $f$ is an $F$-unstable variation of $\Sigma$. Then, by definition, for any variations $x_{\epsilon}$ of $0\in \R^{n+1}$ and $t_{\epsilon}$ of $1$, we have $(F_{x_\epsilon,t_\epsilon}(\Sigma_\epsilon))''\big|_{\epsilon=0}< 0$. In particular, this holds for the trivial variations $x_\epsilon=0$ and $t_\epsilon=1$: $$(F_{0,1}(\Sigma_\epsilon))''\big|_{\epsilon=0}=-\int_\Sigma fLf e^{-\frac{|x|^2}{4}}d\mu<0.$$ 

This shows that $f$ is an unstable variation of $\Sigma$ when $\Sigma$ is considered as critical point to the Gaussian area functional $F$. This shows that the $F$-index is bounded from above by $m-n-2$.

Now we will show that the $F$-index is at least $m-n-2$. By definition of the Morse index, there exist $m$ linearly independent eigenfunctions $u_1,\dots,u_m\in C^{2,\alpha}(\Sigma)$ of $L$ corresponding to the eigenvalues $$\mu_1<\mu_2\leq \dots\leq \mu_m<0.$$ Colding and Minicozzi \cite[Corollary 5.15]{CoMi} have shown that there exists an orthonormal basis of eigenfunctions of $L$ for the weighted $L^2$ space. Therefore, without loss of generality, we may assume that $u_1,\dots,u_m$ are orthonormal in the weighted $L^2$ space. The number of functions in $\{u_1,\dots,u_m\}$ which are orthogonal to $H$ and the components $\mathbf{n}_1,\dots,\mathbf{n}_{n+1}$ of $\mathbf{n}$ is $m-n-2$. Pick any such function $u_i$ (if none exists, the result holds vacuously) and let $\Sigma_{\epsilon}$ be a variation of $\Sigma$ by $u_i\mathbf{n}$. Choose any $y\in \R^{n+1}$ and any $h\in \R$ and set $x_{\epsilon}=\epsilon y$ and $t_{\epsilon}=1+\epsilon h$. Then
\begin{align*}
(F_{x_{\epsilon},t_{\epsilon}}(\Sigma_{\epsilon}))''\big|_{\epsilon=0}
&=\int_{\Sigma}\left(-u_iLu_i-h^2H^2-\frac{1}{2}\langle y,\mathbf{n}\rangle^2\right)e^{-\frac{|x|^2}{4}}d\mu \\
&\leq \int_{\Sigma}\left(-u_iLu_i\right)e^{-\frac{|x|^2}{4}}d\mu \\
&= \mu_i\int_{\Sigma}u_i^2e^{-\frac{|x|^2}{4}}d\mu \\
&<0
\end{align*} by orthogonality and the assumption that $\mu_i<0$. Since there are $m-n-2$ possible linearly independent choices for such a $u_i$, \ref{eq:F-index vs Morse index} follows.

If $\ind_M>n-2$, then Colding-Minicozzi's Theorem \ref{thm:F-unstable implies entropy unstable} implies that $\Sigma$ has positive entropy index, meaning that $\Sigma$ is entropy unstable.
\end{proof}

\section{Entropy Instability of CSF Singularities}\label{S:Entropy Stability of CSF Singularities}

\paragraph{}
Having established in the previous section the relationship between the entropy index and Morse index of a shrinker, we will now calculate the Morse index of $m$-covered circles $\Gamma_m$ and the Abresch-Langer curves $\Gamma_{m,n}$. For $1$-dimensional closed shrinkers, the Jacobi operator is a Sturm-Liouville operator.

\begin{prop}
Let $\Gamma\subset \R^2$ be a shrinker. Then the Jacobi operator $L$ (see (\ref{Eq:L})) of $\Gamma$ is a Sturm-Liouville operator. In particular, for all $j\geq 1$, the eigenfunctions $u_{2j-1}$ and $u_{2j}$ of $L$ have exactly $2j$ zeros.
\end{prop}

\begin{proof}
We refer the reader to \cite[Section 5]{CoMi} for the properties of the Jacobi operator of a shrinker. Differentiate the relation $k=\frac{1}{2}\langle x,\mathbf{n}\rangle $ to obtain $2k_s/k=\langle x,\mathbf{t}\rangle$ and note that $\nabla(\cdot)=\partial_s(\cdot)\mathbf{t}$. Therefore 
\begin{align*}
L
&=\Delta-\frac{1}{2}\langle x,\nabla(\cdot)\rangle +k^2+\frac{1}{2} \\
&=\partial_{ss}-\frac{1}{2}\langle x,\partial_s(\cdot)\mathbf{t}\rangle+k^2+\frac{1}{2} \\
&=\partial_{ss}-\frac{k_s}{k}\partial_s+k^2+\frac{1}{2} \\
&=k\left(\partial_s\left(\frac{1}{k}\partial_{s}\right)+\frac{1}{k}(k^2+\frac{1}{2})\right).
\end{align*}
Since the Abresch-Langer curves are convex (i.e.\ they have $k>0$, see \cite{AbLa}), the latter part of the theorem follows from a general Sturm-Liouville theory \cite[Theorem 8.3.1]{CoLe}.
\end{proof}

\paragraph{}
For $m$-covered circles, the Jacobi operator reduces even further, since the curvature is a constant. In this case, we are able to calculate the spectrum and the entropy unstable variations precisely (see Remark \ref{rmk:entropy unstable variations of circles}).

\begin{thm}\label{thm:entropy index of circles}
Let $\Gamma_m$ be an $m$-covered circle of radius $\sqrt{2}$. Then the $F$-index of $\Gamma_m$ is $$\ind_F(\Gamma_m)=2\lceil \sqrt{2}m\rceil-4,$$ where $\lceil \,\cdot\,\rceil$ is the ceiling function. In particular, $\Gamma_1$ is the only entropy stable circle.
\end{thm}

\begin{proof}
Since $\Gamma_m$ has a constant curvature, $k_s=0$. Thus the Jacobi operator reduces to $$L=\partial_{ss}+1,$$ and the spectrum can be calculated explicitly: $$\mu_j=\frac{j^2}{2m^2}-1$$ for $j\in \Z$. Straightforward calculation using the above formula gives that the number of negative eigenvalues is $2\lceil \sqrt{2}m\rceil-1$. The result now follows from Theorem \ref{thm:equivalence of F and Morse index}. Note that the variation function corresponding to $j=0$ is constant and hence proportional to $k$. Furthermore, the variation functions corresponding to $|j|=m$ are in the span of the component functions of the normal $\mathbf{n}$, since $\mu_{j}=-\frac{1}{2}$ in these cases. 
\end{proof}

\begin{remark}\label{rmk:entropy unstable variations of circles}
The unstable variation functions of the $m$-covered circle can be determined explicitly. A unit speed parameterization $x:[0,2\sqrt{2}\pi m]\to \R^2$ of the $m$-covered circle $\Gamma_m$ of radius $\sqrt{2}$ is given by $$x(\theta)=\sqrt{2}(\cos(\theta/\sqrt{2}),\sin(\theta/\sqrt{2})).$$ Straightforward calculation shows that the unstable variations of $\Gamma_m$ are given by the collection of functions $f_j,g_j:[0,2\sqrt{2}\pi m]\to \R$ defined by \be f_j(\theta)=\sin\left(\frac{j}{\sqrt{2}m}\theta\right) \qquad \text{and}\qquad g_j=\cos\left(\frac{j}{\sqrt{2}m}\theta\right),\ee for $j\in \N$ such that $1\leq j< \sqrt{2}m$ and $j\neq m$. By Theorem \ref{thm:entropy index of circles}, these are all of the $F$-unstable variations. Geometrically, varying $\Gamma_m$ by these functions corresponds to enlarging some circles of $\Gamma_m$, while contracting the others. The CSF will amplify this perturbation.
\end{remark}

\paragraph{}
The proof of Theorem \ref{thm:main theorem existence of perturbation reduce the entropy} now follows easily using Theorem \ref{thm:entropy index of circles}.

\begin{proof}[Proof of Theorem \ref{thm:main theorem existence of perturbation reduce the entropy}]
Theorem \ref{thm:entropy index of circles} shows that an $m$-covered circle is entropy unstable for $m\geq 2$. By Theorem \ref{thm:F-unstable implies entropy unstable}, the $m$-covered circle $\Gamma_m$ can be perturbed to a curve $\Gamma'$ with lower entropy. Moreover, since the perturbation can be chosen to be arbitrarily $C^{\infty}$-small, the perturbed curve $\Gamma'$ can be chosen close enough to $\Gamma_m$ so that the perturbed curve also has turning number $m$. The second part of Theorem \ref{thm:main theorem existence of perturbation reduce the entropy} follows immediately from Theorem \ref{thm:main theorem type I singularity then entropy has lower bound}.
\end{proof}

\paragraph{}
Entropy instability means that we can perturb a shrinker slightly to decrease the entropy. Thus the singularity corresponding to the given shrinker will never occur along the flow starting from the perturbed shrinker. In Section \ref{S:Generic CSF}, our construction of a piecewise CSF for closed immersed curves, whose only singularities are embedded circles and type II singularities, rests upon the fact that the embedded circle is the only entropy stable closed plane shrinker. The latter fact will follow from

\begin{thm}
An Abresch-Langer curve $\Gamma_{m,n}$ has $F$-index
\be \ind_F(\Gamma_{m,n})=2n-5.\ee
\end{thm}

\begin{proof}
Let $L$ be the Jacobi operator of $\Gamma:=\Gamma_{m,n}$. We claim that $k_s/k$ is an eigenfunction of $L$ with eigenvalue 0, i.e. that $L(k_s/k)=0$. Indeed, differentiating the shrinker equation (\ref{eq:shrinker equation}) gives $2k_s/k=\langle x,\mathbf{t}\rangle$. Therefore
\begin{align*}
2L(k_s/k)
&=\partial_{ss}\langle x,\mathbf{t}\rangle-\frac{k_s}{k}\partial_s\langle x,\mathbf{t}\rangle+(k^2+\frac{1}{2})\langle x,\mathbf{t}\rangle \\
&=\partial_{s}(1-k\langle x,\mathbf{n}\rangle)-\frac{1}{2}\langle x,\mathbf{t}\rangle(1-k\langle x,\mathbf{n}\rangle)+(k^2+\frac{1}{2})\langle x,\mathbf{t}\rangle \\
&=-k_s\langle x,\mathbf{n}\rangle-k^2\langle x,\mathbf{t}\rangle+\frac{1}{2}k\langle x,\mathbf{t}\rangle \langle x,\mathbf{n}\rangle+k^2\langle x,\mathbf{t}\rangle \\
&=-k_s\langle x,\mathbf{n}\rangle+k_s\langle x,\mathbf{n}\rangle \\
&=0.
\end{align*}
Since $k$ is strictly positive and $k_s$ has exactly $2n$ zeros on $[0,2\pi m)$, $k_s/k$ is either the $(2n-1)^{\text{st}}$ or $2n^{\text{th}}$ eigenfunction of $L$. It remains to show that $k_s/k$ is the $(2n-1)^{\text{st}}$ eigenfunction and that zero is a simple eigenvalue. To do so, we reparameterize $\Gamma$ using the variable 
\[\theta=-\inv{\cos}\langle e_1,\mathbf{n}\rangle,\] 
where $e_1\in \R^2$ is a constant unit vector. This is possible since $\Gamma$ is convex (however, $\theta$ is only locally continuous as a discontinuity appears after one round). It follows from the shrinker equation that if we define the operator $\tilde{L}$ by \[\tilde{L}=k^2\partial_{\theta\theta}+(k^2+1/2),\] then $\frac{\partial \theta}{\partial s}=k$ and $\tilde{L}f=Lf$. The theorem now follows from the linear analysis in \cite[Proposition 2.1]{EpWe}, however, we provide a proof here for the convenience of the reader.

Note that $k_{\theta}=k_s/k$ and \[\tilde{L}k_{\theta}=\tilde{L}(k_s/k)=L(k_s/k)=0.\] In other words, $k_{\theta}$ is an eigenfunction of $\tilde{L}$ with eigenvalue $0$. We denote by $\{\mu_j\}$ and $\{\nu_j\}$, respectively, the Dirichlet and Neumann eigenvalues of $\tilde{L}$ on $[0,\pi m/n]$. (By our conventions, these sequences are nondecreasing.) Let $f_j$ solve the Neumann problem $\tilde{L}f_j=-\nu_jf_j$ on $[0,\pi m/n]$. Since $k$ is an even function, $f_j$ may be extended by reflection and then periodically to an eigenfunction $\overline{f}_j$ of $\tilde{L}$ on the circle of length $2\pi m$ with the same eigenvalue $\nu_j$. Precisely, the function $\overline{f}_j$ is defined on $[0,2\pi m/n]$ by
\be
\overline{f}_j(\theta)=
\begin{cases}
f_j(\theta) \qquad\qquad \;\; \text{for } \theta\in \left[0,\frac{\pi m}{n} \right] \\
f_j\left(\frac{2\pi m}{n}-\theta\right) \quad \text{for } \theta\in\left[\frac{\pi m}{n} ,\frac{2\pi m}{n} \right],
\end{cases}
\ee
and is extended periodically to the circle of length $2\pi m$.
The eigenfunction $\overline{f}_2$ will have exactly $2n$ zeros on $[0,2\pi m)$ and $k_{\theta}$ is the lowest Dirichlet eigenfunction of $\tilde{L}$ on $[0,\pi m/n]$, so $\mu_1=0$. From the standard fact that $\mu_1\geq \nu_2$, it follows that $k_{\theta}$ and $\overline{f}_2$ are, respectively, the $(2n-1)^{\text{st}}$ and $2n^{\text{th}}$ eigenfunctions of $\tilde{L}$ on $[0,2\pi m]$. Moreover, since $\tilde{L}$ and $L$ have the same eigenfunctions, $k_s/k=k_{\theta}$ is the $(2n-1)^{\text{st}}$ eigenfunction of $L$.

Next we show that zero is a simple eigenvalue. Note that $\tilde{L}$ is a linear second order differential operator so the equation $\tilde{L}f=0$ can have at most two linearly independent solutions. If zero were not a simple eigenvalue, then there would exist a nontrivial $2\pi m$ periodic solution $w$ of $\tilde{L}w=0$. Consequently, any solution of $\tilde{L}f=0$, being a linear combination of $k_{\theta}$ and $w$, which are both $2\pi m$ periodic, would have to be $2\pi m$ periodic. As a result, to show that zero is a simple eigenvalue, it is sufficient to produce a solution to $\tilde{L}f=0$ which is not $2\pi m$ periodic.

We will now produce such a solution. It follows from the shrinker equation that $k$ solves the ODE 
\be\label{eq:shrinker ODE}
\qquad k_{\theta\theta}+k-\frac{1}{2k}=0.
\ee 
This equation has first integral \[ E=k_{\theta}^2+k^2-\log k\] and the general solution can be expressed as $k(\theta+a,E)$. Let \[u=\frac{\partial k}{\partial E}\bigg|_{E=E_{m,n}},\] where $E_{m,n}$ is the constant corresponding to $\Gamma_{m,n}$. Straightforward calculation using (\ref{eq:shrinker ODE}) shows that $\tilde{L}u=0$. Abresch-Langer \cite[Proposition 3.2]{AbLa} show that $u$ is not $2\pi m$ periodic. As stated above, it follows that zero is a simple eigenvalue of $\tilde{L}$ and thus also of $L$. 

Summarizing, we have shown that $L(k_s/k)=0$, that $k_s/k$ is the $(2n-1)^{\text{st}}$ eigenfunction of $L$, and that zero is a simple eigenvalue of $L$. Consequently, $L$ has  $2n-2$ negative eigenvalues. By Theorem \ref{thm:equivalence of F and Morse index}, the Abresch-Langer curve $\Gamma_{m,n}$ has $F$-index $2n-5$.
\end{proof}

\begin{cor}\label{cor:stability of closed plane shrinkers}
The embedded circle is the only entropy stable closed shrinker of the CSF for closed immersed curves.
\end{cor}

\section{Generic CSF}\label{S:Generic CSF}

\paragraph{}
In this section, we prove Theorem \bref{thm:PCSF}{B}. The following lemma shows that, as a consequence of Corollary \ref{cor:stability of closed plane shrinkers}, we can perturb an unstable closed plane shrinker to satisfy the entropy conditions (\ref{eq:entropy condition 1}) for a piecewise CSF. For the piecewise flow in Theorem \bref{thm:PCSF}{B}, we will actually require the entropy inequality (\ref{eq:entropy condition 1}) to be a strict inequality so that we can exclude a shrinker from appearing as a singularity at later times of the flow. 

\begin{lemma}\label{lem:PCSF is possible}
Let $\Gamma\subset \R^2$ be a closed immersed curve such that the CSF $(\Gamma_t)_{t\in [0,T)}$ starting from $\Gamma$ has a type I singularity other than the embedded circle and let $\tilde{\Gamma}_{\infty}\subset \R^2$ be some blowup sequence limit at this singularity. Then there exists a $T_0\in[0,T)$ and a function $u\in C^{\infty}(\Gamma)$ such that the graph $\bar{\Gamma}=\Gamma_{T_0}+u\mathbf{n}$ satisfies 
\begin{align}
\lambda(\bar{\Gamma})&< \lambda(\tilde{\Gamma}_{\infty}). \label{eq:entropy condition}
\end{align}
\end{lemma}

\begin{proof}
If the CSF $(\Gamma_t)_{t\in [0,T)}$ has a type I singularity at time $T$, any limit of a rescaling sequence is a closed shrinker and thus is compact. Therefore, there is only one singular point $x_1\in \R^2$ at time $T$. Without loss of generality, we may assume that $0\in \R^2$ is a singular point. 
By the classification of 1-dimensional shrinkers \cite{AbLa}, $\tilde{\Gamma}_{\infty}$ must be a multiply-covered circle or a multiply-covered Abresch-Langer curve (including $1$ time covered). Since we assume that the shrinker $\tilde{\Gamma}_{\infty}$ is not an embedded circle, Corollary \ref{cor:stability of closed plane shrinkers} implies that $\tilde{\Gamma}_{\infty}$ is entropy unstable. 
By Theorem \ref{thm:entropy achieved in compact set}, there exists a variation $f\in C^{\infty}(\tilde{\Gamma}_{\infty})$ and an $\epsilon_0>0$ such that \[\lambda(\tilde{\Gamma}_{\infty}+\epsilon f\mathbf{n})<\lambda(\tilde{\Gamma}_{\infty}),\] for all $\epsilon \neq 0$ with $|\epsilon|<\epsilon_0$. By assumption, there exists a sequence $\tau_i\to \infty$ of rescaled times such that the rescaled curves $(\tilde{\Gamma}_{\tau_i})_{i\in \N}$ smoothly converge to $\tilde{\Gamma}_{\infty}$. Thus, there exists $N'\in \N$ and a sequence $(g_i)_{i\in \N}\subset C^{\infty}(\tilde{\Gamma}_{\infty})$ with $\|g_i\|_{C^{\infty}}\to 0$ such that, for $i>N'$, the rescaled curve $\tilde{\Gamma}_{\tau_i}$ can be written as a graph \[\tilde{\Gamma}_{\tau_i}=\tilde{\Gamma}_{\infty}+g_i \mathbf{n}.\] By the monotonicity of entropy and Corollary \ref{cor:continuity of entropy near a shrinker}, there exists $\epsilon_1\leq \epsilon_0$ and $N\geq N'$ such that, for all $\epsilon\neq 0$ with $|\epsilon|<\epsilon_1$ and all $i>N$, we have \[ \lambda(\tilde{\Gamma}_{\infty}+(\epsilon f+g_i)\mathbf{n})<\lambda(\tilde{\Gamma}_{\infty})\leq \lambda(\tilde{\Gamma}+g_i\mathbf{n}).\] 

Since entropy is invariant under dilations, when the curve $\tilde{\Gamma}_{\infty}+(\epsilon f+g_i)\mathbf{n}$ is rescaled back to the original spacetime, the entropy conditions can be satisfied when $|\epsilon|\neq 0$ is small enough, $i$ is large enough. The time $T_0$ then corresponds to the rescaled time $\tau_i$ and the function $u$ is a multiple of the function $\epsilon f$.
\end{proof}

\paragraph{}
By using the classification of singularities for the CSF of closed curves and Lemma \ref{lem:PCSF is possible}, we now prove Theorem \bref{thm:PCSF}{B}.

\begin{proof}[Proof of Theorem \bref{thm:PCSF}{B}]
We will construct a piecewise CSF with a finite number of discontinuities that eventually shrinks to a circular singularity, or has a type II singularity. We perform a smooth jump just before a (entropy unstable) singular time, where we replace a time slice of the flow by a graph over it, and the crucial point is to show that the entropy decreases below the entropy of the unstable singularity. We repeat this until we get a singular point which is either modelled by an embedded circle or type II.

Let $(\Gamma_t)_{t\in [0,T_1)}$ be the CSF starting at $\Gamma$. The CSF $(\Gamma_t)_{t\in [0,T_1)}$ has either a type I or a type II singularity at time $T_1$. We consider both cases:
\setlist[description]{font=\normalfont\itshape\space}
\begin{description}
\item[Case (I):]
If the CSF $(\Gamma_t)_{t\in [0,T_1)}$ has a type I singularity at time $T_1$, any limit of a rescaling sequence is either an entropy stable shrinker or an entropy unstable closed shrinker. In particular, any limit of a rescaling sequence is compact, and therefore, there is only one singular point $x_1\in \R^2$ at time $T_1$.
\begin{description}
\item[Case (I.a):] 
If there is some rescaling sequence limit which is an entropy unstable closed shrinker $\Gamma_{\infty}^1$, Lemma \ref{lem:PCSF is possible} gives the existence of a $t_2<T_1$ and a curve $\Gamma_{t_2}^2$ such that $\Gamma_{t_2}^2$ is a graph over $\Gamma_{t_2}^1:=\Gamma_{t_2}$ of a function $u_2\in C^{2,\alpha}(\Gamma_{t_2}^1)$ satisfying
\begin{align*}
\lambda(\Gamma_{t_2}^2)&< \lambda(\Gamma_{\infty}^1). 
\end{align*}
Consequently, there exists a CSF $(\Gamma_{t}^2)_{t\in [t_2,T_2)}$ starting at $\Gamma_{t_2}^2$ such that the concatenation of $(\Gamma_t^1)_{t\in [0,t_2]}$ with $(\Gamma_t^2)_{t\in [t_2,T_2)}$ is a piecewise CSF starting at $\Gamma$.
\item[Case (I.b):]
If all rescaling sequences converge to an entropy stable closed shrinker, it must be an embedded circle, so the theorem holds.
\end{description}
\item[Case (II):]
If the CSF $(\Gamma_t)_{t\in [0,T_1)}$ has a type II singularity at time $T_1$, the theorem holds.
\end{description}
If {\it Case (I.a)} applies, we return to cases {\it (I)} or {\it (II)} for $(\Gamma_t^2)_{t\in [t_2,T_2)}$, and repeat the process. Since the piecewise CSF preserves turning number and there are only finitely many closed planar shrinkers of a given turning number, by the fact that \[\lambda(\Gamma_{T_i}^i)<\lambda(\Gamma_{\infty}^{i-1})\] at each step and the monotonicity of entropy under piecewise CSF, {\it Case (I.a)} can happen only finitely many times. This proves the first part of the theorem.

It was shown in the proof of Theorem \ref{thm:main theorem type I singularity then entropy has lower bound} that any blowup sequence of a type I singularity of the CSF for closed curves preserves the turning number. It follows that a closed curve with turning number greater than 1 cannot have an embedded circle as the limit of a rescaling sequence. This proves the second part of Theorem \bref{thm:PCSF}{B}.
\end{proof}


{\it J.B.'s address: Department of Mathematics, Massachusetts Institute of Technology, Cambridge, MA 02139, USA. email: juliusbl@mit.edu}

{\it A.S.'s address: Department of Mathematics, Massachusetts Institute of Technology, Cambridge, MA 02139, USA.}

{\it A.S.'s current address: Department of Mathematics,
	University of Chicago,
	5734 S. University Avenue,
	Chicago, IL, 60637. email: aosun@uchicago.edu.}

\end{document}